\documentclass[10pt]{amsart}
\usepackage{bbm}

\newtheorem{thm}{Theorem}[subsection]
\newtheorem{cor}[thm]{Corollary}
\newtheorem{lem}[thm]{Lemma}

\newtheorem{definition}[thm]{Definition}
\newtheorem{example}[thm]{Example}
\newtheorem{remark}[thm]{Remark}

\numberwithin{equation}{subsection}

\newcommand{\R}{\mathbb{R}}  
\newcommand{\Z}{\mathbb{Z}}
\newcommand{\F}{\mathbb{F}}
\newcommand{\Q}{\mathbb{Q}}

\newcommand{\D}{\mathbf{D}}
\newcommand{\V}{\mathbf{V}}

\begin{document}


\title{On Certain Computations of Pisot Numbers}


\author{Qi Cheng}
\address{School of Computer Science, The University of Oklahoma,
Norman, OK 73019} \email{qcheng@cs.ou.edu}
\author{Jincheng Zhuang}
\address{School of Computer Science, The University of Oklahoma,
Norman, OK 73019} \email{jzhuang@ou.edu}

\thanks{}



\begin{abstract}
This paper presents two algorithms on certain computations about
Pisot numbers. Firstly, we develop an algorithm that finds a Pisot
number $\alpha$  such that $\Q[\alpha] = \F$ given a real Galois
extension $\F$ of $\Q$ by its integral basis.
This algorithm is  based on the lattice reduction, and it runs in
time polynomial in the size of the integral basis. Next, we show
that for a fixed  Pisot number $\alpha$, one can compute $
[\alpha^n] \pmod{m}$ in time polynomial in  $(\log (m n))^{O(1)}$,
where $m$ and $n$ are positive integers.

\end{abstract}


\maketitle



\section*{Introduction}

A Pisot number (or Pisot-Vijayaraghavan number) is a real algebraic
integer greater than 1, whose Galois conjugates over $\Q $ are all
of modulus strictly less than 1. Generally, given a real number
$\varepsilon > 0,$ an algebraic integer is called an
$\varepsilon-$Pisot number if all its Galois conjugates have modulus
less than $\varepsilon$ \cite{FS04}. The most famous Pisot number is
the golden radio $ 1+ \sqrt{5} \over 2$. Pisot numbers have many
interesting properties in their own right. Not surprisingly, they
have many applications in diverse areas, such as harmonic analysis,
statistics and the Diophantine approximation. For an introduction to
the Pisot numbers, we refer the reader to the books \cite{SL63} and
\cite{PS92}.

In this paper, we study two computational problems about Pisot
numbers: one is to find  a Pisot number generating a real Galois
number field, and the other is to compute the modular exponentiation
of a Pisot number.


There are several known ways to find Pisot numbers in different
situations. To name a few: Dufresnoy and Pisot \cite{DP55} developed
a method to find all Pisot numbers in the real interval
$[1,\frac{1+\sqrt{5}}{2}+\varepsilon],$ where
$0<\varepsilon<0.0004.$ Boyd \cite{Boyd78} modified Dufresnoy and
Pisot's algorithm to determine all the Pisot numbers in an interval
of the real line $[\alpha,\beta]$ if there are finitely many in the
interval. Bell and Hare \cite{BH05} gave a classification of some
Pisot-Cyclotomic numbers. Utilizing the
Lenstra-Lenstra-Lovasz (LLL) algorithm \cite{LL82}, we show the following result.

\begin{thm}
Let $\F$ be a real Galois extension over $\Q$
given by its integral basis
$\beta_1,\cdots,\beta_k $. There exists a polynomial time algorithm
to determine integers $a_1,a_2, \cdots, a_k$ such that
$$\alpha = a_1 \beta_1 + \cdots + a_k \beta_k $$
is a Pisot number and $\Q[\alpha] = \F$.
\end{thm}

\begin{remark}
There are many ways to represent an algebraic number \cite{Cohen96}.
For example, one can represent an algebraic number by its minimal
polynomial and a complex number, which is closer to the number than
any of its conjugates. The size of an algebraic number is defined to
be the size of its minimal polynomial. The size of an integral
polynomial $\sum_{i=0}^d a_i x^d \in \Z[x]$ ($a_d \not=0$) is
defined to be $d \log (\max_i \{|a_i|+1\})$. In many of our
examples, we work in the real sub-field of a cyclotomic field
$\Q(\zeta)$, where $\zeta$ is a primitive root of unity. This allows
us to  represent an algebraic number as an element in $\Q[\zeta]$.
\end{remark}

\begin{remark}
For many number fields,  integral bases are known. However,
computing an integral basis of a number field is, in general, not an
easy problem, as it involves factorization of a large integer
\cite{Cohen96}.
\end{remark}

It is well-known that for a Pisot number $\alpha$, $\alpha^n$ is
exponentially close to an integer as $n$ grows. In this paper, we
investigate the problem of computing the integer $[\alpha^n]$ and
its remainder  modulo a positive integer $m$, where $[*]$ denotes
the function to the nearest integer. Modular exponentiation is the
most important operation in implementation of a public key
cryptography. By using the repeated squaring algorithm, $a^n$ can be
computed using only $O(\log n)$ many multiplications, and hence, $
a^n \pmod{m}$ can be computed efficiently if $a$ is an integer.
However, if the base of the exponentiation $a$ is not an integer,
then the problem of computing $ [a^n] $ is considered to be hard.
Note that $ [a^n] $ can be too large to be outputted, but in many
cases, we are interested in the number of basic operations in
integers to produce the number, regardless to the size of the
operands. To this end, Tau functions were introduced to measure
 the complexity of  an integer \cite{BSS}.

\begin{definition} A straight-line program to compute an integer
$n\in{\mathbb{N}}$ is a sequence of ring operations (namely,
addition, subtraction and multiplication) to produce the integer $n$
from the constant $1$. Let $\tau(n)$ be the length of the shortest
straight-line program computing $n$. For a sequence of integers
$x_1, x_2, \cdots, x_i, \cdots $, if there exists a polynomial $p$
such that $\tau(x_n)\leq{p(\log{n})}$, then the sequence of integers
is called easy to compute. Otherwise, we say that the sequence is
hard to compute.
\end{definition}

Many well-known integer sequences are conjectured to be hard to
compute, e.g. $n!$. Pascal Koiran \cite{PK04} conjectured that the
sequences $\lfloor 2^n\sqrt{2} \rfloor $  and $ \lfloor (3/2)^n
\rfloor $ are also hard to compute. Here we show that, on the
contrary, a similar sequence $[\alpha^n]$ is easy to compute if
$\alpha$ is a Pisot number. Namely, we show the following:

\begin{thm}
For a fixed  Pisot number $\alpha$, we can find a straight-line
program of length $O(\log n )$ for $[\alpha^n]$ in time $(\log
n)^{O(1)}$. Hence,
$$ \tau ([\alpha^n]) = O( \log n ). $$
\end{thm}

As a corollary, we prove that
the problem of computing the modular exponentiation
of a Pisot number is easy.
More precisely,

\begin{cor}
Given a Pisot number $\alpha$, and two positive integers $m$ and $n$,
there exists an algorithm to compute $[\alpha^n]\mod{m}$
in time $(\log (mn))^{O(1)}$.
\end{cor}

The paper proceeds as follows: Section 1 demonstrates the first
algorithm to determine a Pisot number generating a given real algebraic field
and proves the Theorem 0.0.1. Section 2 describes the algorithms
to find a straight-line program for $[\alpha^n]$ and to compute
$[\alpha^n]\mod{m}$ of a given Pisot number $\alpha$ and proves the
Theorem 0.0.5 and Corollary 0.0.6.

\textbf{Notations:} Let the lowercase letters in bold and the capital letters in bold represent vectors and matrices, respectively.

\section{An algorithm to search a Pisot number in a totally real number field}
\subsection{Preliminaries}
Let $\R^n$ be the $n-$dimensional Euclidean space. A (full rank)
integral lattice is the set
$$\mathcal{L}=\{\sum_{i=1}^{n}x_i \mathbf{b}_i|x_i\in\Z\},$$
where $\mathbf{b}_1,\mathbf{b}_2,\cdots,\mathbf{b}_n$ are
linearly independent vectors over $\R$ and $\mathbf{b}_i \in \Z^n$
for $1\leq i \leq n$.
The determinant of the lattice is defined to be the absolute value
of the determinant of the matrix $(b_{ij})$, where $b_{ij}$ is the
$j-$th coordinate of $\mathbf{b}_i.$

Minkowski's convex body theorem (page 12 in \cite{MG02}) asserts
that given any convex set in $\R^n$, which is symmetric with respect
to the origin and with volume greater than $2^n\det(\mathcal{L})$, there
exists a non-zero lattice point in the set. As a corollary,
Minkowski's first theorem says that the length of the shortest
vector in $\mathcal{L}$ satisfies
$$\lambda_1<\sqrt{n}\det(\mathcal{L})^{1/n}.$$

While no known efficient algorithm can find the shortest vector,
or even a vector within the Minkowski's bound,
there are  polynomial time algorithms to approximate
the shortest vector in a lattice.
The Lenstra-Lenstra-Lovasz (LLL) algorithm can find in polynomial
time a vector whose length is at most $(2/\sqrt{3})^n$ times the
length of  the shortest vector of a lattice (see page 33 in
\cite{MG02}). The Block-Korkine-Zolotarev (BKZ) algorithm can
achieve a better approximation factor. In this paper, we use the LLL
reduction algorithm, which is adequate for our purpose.

\subsection{The problem and the idea}
Let $\F$ be a real algebraic field and let
$\beta_1,\cdots,\beta_k$ be its integral basis. Each algebraic
integer in $\F$ can be represented as
$$\alpha=z_1\beta_1+\cdots+z_k\beta_k,$$
where $z_1, z_2, \cdots, z_k $ are rational integers, and its
conjugates are
$$\sigma_i(\alpha)=z_1\sigma_i(\beta_1)+\cdots+z_k\sigma_i(\beta_k), (1\leq i \leq k-1),$$
where each $\sigma_i (1\leq i \leq k-1)$ is a field automorphism of
$\F$.

Let us consider the lattice $\mathcal{L}$ generated by the $k$
column vectors of the following matrix:
$$\D=\begin{bmatrix}
  \beta_1 & \beta_2 & \cdots & \beta_k     \\
  \sigma_1(\beta_1) & \sigma_1(\beta_2) & \cdots & \sigma_1(\beta_k)     \\
  \vdots   &           \vdots            &\ddots   & \vdots \\
  \sigma_{k-1}(\beta_1) & \sigma_{k-1}(\beta_2) & \cdots &
  \sigma_{k-1}(\beta_k)
\end{bmatrix}
.$$ We note that the square of the determinant of the matrix $\D$ is
the discriminant of the field $\F.$ Each column of $\D$ consists of one element of the integral basis and
its conjugates, thus each vector in the
lattice $\mathcal{L}$ corresponds to an algebraic integer of $\F$ given by
its first element.

It can be proved that there exist Pisot numbers in the field $\F$ by
applying Minkowski's theorem on the lattice $\mathcal{L}$ \cite{Pisot38} and
\cite[Page 3]{SL63}. Furthermore, we can derive an upper bound of
the minimal Pisot number from the above proof. For the completeness,
we include the modified proof below:

\begin{lem} \cite{SL63} Let $\F$ be a real algebraic
field with discriminant $\Delta_{\F}$. Given a real number $0<
\delta < 1,$ there exists a Pisot number $\alpha$ bounded by
$B=\frac{\sqrt{|\Delta_{\F}|}}{\delta^{k-1}}$ such that $\Q[\alpha]
= \F$.
\end{lem}

\begin{proof}
Firstly, we show the existence of Pisot numbers in the field.
For any positive real number $B$ and $\delta < 1 $,
 all the points $(r_1, r_2, \cdots, r_k) \in \R^k$ satisfying
$$ |r_1| <B {\rm \ \
and\ \ }|r_i| < \delta, (1\leq i \leq k-1) $$ form a hyper-cuboid of
volume $2^k B \delta^{k-1}$. If $B\delta^{k-1} \geq
\sqrt{|\Delta_{\F}|}$, which can be satisfied by set
$B=\frac{\sqrt{|\Delta_{\F}|}}{\delta^{k-1}}$. Then by Minkowski's
convex body theorem, there exists a nonzero lattice point of $\mathcal{L}$ in
the convex body. In other words, there exists rational integers
$z_1,\cdots,z_k$ such that
$$|z_1\beta_1+\cdots+z_k\beta_k|\leq B,$$
$$|z_1\sigma_i(\beta_1)+\cdots+z_k\sigma_i(\beta_k)|\leq \delta <1
\ (1\leq i \leq k-1).$$
Hence, the algebraic integer
$$\alpha=z_1\beta_1+\cdots+z_k\beta_k$$
is a Pisot number by definition.

Next, we need to show that $\Q[\alpha] = \F$. By definition,
$\alpha$ is greater than any of its conjugates. If we denote
$e=[\Q[\alpha]:\Q], f=[\F:\Q]$ and suppose $e<f$, then we have $e|f$
and let $f=ed$ for some $d\in \Z$. Thus, $\alpha$ will appear $d$ times in its Galois
conjugates, which is a contradiction with the definition of the
Pisot number.
\end{proof}

\subsection{The algorithm and its correctness}

The above lemma demonstrates the existence of the Pisot number in a
real algebraic field. However, the proof is nonconstructive; it does
not provide an efficient method to find a Pisot number. The key idea
of our algorithm below is to construct a new lattice similar to $\mathcal{L}$
and to convert the problem of determining a Pisot number in a given
total real field into the problem of finding a vector in the
lattice, whose length approximates the shortest vector.

Let $P$ be a positive real number,  we construct another lattice $\mathcal{L}_P$ generated
by the $k$ column vectors of the following matrix:

\begin{equation}\label{LP}
\D_P=\begin{bmatrix}
  \beta_1 & \beta_2 & \cdots & \beta_k     \\
  P\sigma_1(\beta_1) & P\sigma_1(\beta_2) & \cdots & P\sigma_1(\beta_k)     \\
   \vdots &  \vdots                       &\ddots   &\vdots  \\
  P\sigma_{k-1}(\beta_1) & P\sigma_{k-1}(\beta_2) & \cdots &
  P\sigma_{k-1}(\beta_k)
\end{bmatrix}
.
\end{equation}
Note that
$$\D = \D_1 {\rm \ and\ \ }  \det(\D_P) = P^{k-1} \det(\D). $$
We observe:
\begin{itemize}
\item From Minkowski's Theorem, we conclude that
%
there is vector in $\mathcal{L}_P$ with length at most $ \sqrt{k}
\sqrt[k]{P^{k-1} \det(\D)}$;
\item
On the other hand, if a vector is not
corresponding to a Pisot number, then its length is at least
$P$.
\end{itemize}
So we can choose an appropriate $P$ such that the gap between
$\sqrt{k} \sqrt[k]{P^{k-1} \det(\D)}  $ and $P$ is large enough, then
the LLL algorithm can find a short vector of length less than $P$,
which must correspond to a Pisot number. Our algorithm can be
described as follows:


\vskip 0.2cm

 \noindent \framebox[4.8in]{
\begin{minipage}{4.7in}
{\bf Algorithm 1}

Input: Integral basis $\beta_1,\cdots,\beta_k$ of a real Galois extension
$\F$ over $\Q$.

\begin{enumerate}
\item Compute $P$ to be an integer bigger than $(\frac{2}{\sqrt{3}})^{k^2}k^{k/2}\det(\D)$;
\item Construct the basis of the lattice $\mathcal{L}_P$ as the columns of the matrix $\D_P$ defined by Equation (\ref{LP});
\item Run LLL algorithm on the basis of $\mathcal{L}_P$;
\item Recover the Pisot number which is the absolute value of the first element of the returned approximate shortest vector.
\end{enumerate}
Output: A Pisot number.
\end{minipage}
}

\vskip 0.2cm

Now we proceed to prove Theorem 0.0.1. We need to show that the proposed
algorithm is correct and it runs in polynomial time of the input
size.

\vskip 0.2cm


\begin{proof} (of Theorem 0.0.1)
Firstly, we need to show that this algorithm returns a Pisot number.
According to Minkowski's Theorem, there is a nonzero vector in $\mathcal{L}_P$
of length at most $ \sqrt{k} \sqrt[k]{P^{k-1} \det(\D)}$. On the
other hand, for any algebraic integer $x_1\beta_1+\cdots+x_k\beta_k$
in $\F$ that is not a Pisot number, the vector
$$ x_1 \left( \begin{matrix} \beta_1\\ P\sigma_1(\beta_1)\\ \vdots
          \\ P\sigma_{k-1}(\beta_1) \end{matrix} \right)
+  x_2 \left( \begin{matrix} \beta_2\\ P\sigma_1(\beta_2)\\ \vdots
          \\ P\sigma_{k-1}(\beta_2) \end{matrix} \right)
+ \cdots + x_k \left( \begin{matrix} \beta_k\\ P\sigma_1(\beta_k)\\ \vdots
          \\ P\sigma_{k-1}(\beta_k) \end{matrix} \right)$$
in the lattice has length at least $P$, since
 there exists $1\leq i\leq k-1$ such
that
$$ |x_1P\sigma_i(\beta_1)+\cdots+x_kP\sigma_i(\beta_k)| = P |x_1\sigma_i(\beta_1)+\cdots+x_k\sigma_i(\beta_k)| \geq P. $$
If we set $ P > (\frac{2}{\sqrt{3}})^{k^2}k^{k/2}\det(\D)$, then
$$ P/(  \sqrt{k} \sqrt[k]{P^{k-1} \det(\D)}) >  (\frac{2}{\sqrt{3}})^k,  $$
then the vector returned by the LLL algorithm will have length less
than $P$, which must correspond to a Pisot number.

Next, we need to show that the returned Pisot number $\alpha$ is a
primitive element of the field. By definition, $\alpha$ is greater
than any of its conjugates. If we denote $e=[\Q[\alpha]:\Q],
f=[F:\Q]$ and suppose $e<f$, then we have $e|f$ and let $f=ed$ for some $d\in \Z$. Thus
$\alpha$ will appear $d$ times in its conjugates, which is a
contradiction with the definition of the Pisot number.

At last, we analyze the running time of the algorithm. First we
observe that $\log P = (k \det(\D))^{O(1)}$. The most costly part of
the algorithm is Step 3, where the LLL algorithm runs in polynomial
time of the size of the lattice basis.
Thus the overall time of the algorithm is
polynomial in the size of the integral basis.

\end{proof}

\begin{remark}
Given a real number $\varepsilon > 0$, if we choose $P$ such that
$$P>(\frac{2}{\sqrt{3}})^{k^2}k^{k/2}\det(\D)/{\varepsilon^k},$$
then we have
$$\varepsilon P > (\frac{2}{\sqrt{3}})^{k^2}k^{k/2}\det(\D).$$
Hence, the algorithm actually determines an $\varepsilon-$Pisot
number in this case.
\end{remark}

\subsection{Examples}
In the following examples, we use the lattice functions in Victor
Shoup's NTL package.

\begin{example}
Let us illustrate our algorithm by taking the field
$\Q(2\cos\frac{2\pi}{15})$ as an example. The extension degree
$k=[\Q(2\cos\frac{2\pi}{15}):\Q]=\frac{\phi(15)}{2}=4$ and an
integral basis is given by
$$\beta_1=2\cos\frac{2\pi}{15},\beta_2=2\cos\frac{4\pi}{15},\beta_3=2\cos\frac{8\pi}{15},\beta_4=2\cos\frac{14\pi}{15}.$$

\begin{enumerate}
\item Choose $\varepsilon=0.5$, compute $P=85769>(\frac{2}{\sqrt{3}})^{16}*4^2*\sqrt{1125}*16$;
\item Construct the basis of $\mathcal{L}_P$ as the column vectors of $\D_P$;
\item Run LLL algorithm over the basis of $\mathcal{L}_P$;
\item Recover the Pisot number which is the following:
$$\alpha=2105\beta_1+1215\beta_2+1440\beta_3+139\beta_4.$$
\end{enumerate}
\end{example}

\begin{remark}
We note that in the second step we first compute
$P\sigma_i(\beta_j)$ then take the integer part as the input of the
matrix. And we can check that the Galois conjugates of the returned
number are: -0.063765..., 0.065726..., and -0.048703....
\end{remark}

\begin{example}
Now let's look at another example, the field
$\Q(2\cos\frac{2\pi}{17}).$ The extension degree
$k=[\Q(2\cos\frac{2\pi}{17}):\Q]=\frac{\phi(17)}{2}=8$ and one
integral basis is given by
$$\beta_1=2\cos\frac{2\pi}{17},\beta_2=2\cos\frac{4\pi}{17},\beta_3=2\cos\frac{6\pi}{17},\beta_4=2\cos\frac{8\pi}{17},$$
$$\beta_5=2\cos\frac{10\pi}{17},\beta_6=2\cos\frac{12\pi}{17},\beta_7=2\cos\frac{14\pi}{17},\beta_8=2\cos\frac{16\pi}{17}.$$

\begin{enumerate}
\item Compute $P=825982306366 > (\frac{2}{\sqrt{3}})^{64}*8^4*\sqrt{410338673}$;
\item Construct the basis of $\mathcal{L}_P$ as the column vectors of $\D_P$;
\item Run LLL algorithm over the basis of $\mathcal{L}_P$;
\item Recover the Pisot number which is the following:
$$
\begin{aligned}
\alpha=&-24708871\beta_1-95498414\beta_2-202808109\beta_3-332145187\beta_4\\
       &-466041959\beta_5-586414924\beta_6-677007046\beta_7-725583357\beta_8.
\end{aligned}
$$
\end{enumerate}

\end{example}

\begin{remark}
We can compute the Galois conjugates of the returned number are: 0.039500..., 0.048267..., 0.064900..., -0.019990..., -0.057987...,
0.062209... and 0.036031....

\end{remark}

\section{An algorithm to compute modular exponential of a Pisot number}
\subsection{The problem and the idea}
Given a Pisot number $\alpha$ of degree $d$
and its minimal polynomial over $\Q$
$$ f(x)=x^d+c_{d-1}x^{d-1}+\cdots+c_1x+c_0,$$
we want to determine a straight-line program for $[\alpha^n]$ and
then to compute $[\alpha^n]\mod{m}$, where $n, m$ are given positive
numbers.

\begin{lem} Given a Pisot number $\alpha_1$ of degree $d$ with
conjugates
$\alpha_2,\cdots,\alpha_d,|\alpha_2|\geq|\alpha_i|,3\leq{i}\leq{d}$.
If $n>\log_{|\alpha_2|}{\frac{1}{2(d-1)}}$, then
$$[\alpha_1^n]= [\alpha_1^n+\alpha^n_2+\cdots+\alpha_d^n]= \alpha_1^n+\alpha^n_2+\cdots+\alpha_d^n.$$
\end{lem}

\begin{proof} Suppose $n>\log_{|\alpha_2|}{\frac{1}{2(d-1)}},$ we have
$$
|\alpha_2^n+\cdots+\alpha_d^n|\leq(d-1)|\alpha_2|^n
                            <\frac{1}{2}.
$$
Note that for any given positive integer $n$,
$\alpha_1^n+\alpha_2^n+\cdots+\alpha_d^n$ is an integer itself. Thus
we deduce if $n>\log_{|\alpha_2|}{\frac{1}{2(d-1)}},$ then
$$[\alpha_1^n]=[\alpha_1^n+\alpha^n_2+\cdots+\alpha_d^n]
=\alpha_1^n+\alpha^n_2+\cdots+\alpha_d^n.$$
\end{proof}

Lemma 2.1.1 shows that we can convert the problem of finding $[\alpha^n]$ of a Pisot number $\alpha$ to
the computation of $\alpha^n+\alpha^n_2+\cdots+\alpha_d^n$ when
$n>\log_{|\alpha_2|}{\frac{1}{2(d-1)}}$, where $\alpha_2, \alpha_3,
\cdots, \alpha_d$ are conjugates of $\alpha$. For the sake of
consistency, we will sometimes write $\alpha_1$ in the place of
$\alpha$ below.

\subsection{Notations and preliminaries}
Let the polynomial $f(x)=x^d+c_{d-1}x^{d-1}+\cdots+c_1x+c_0$ be the
minimal polynomial for a Pisot numuber $\alpha$ over $\Q$.
The companion matrix \cite{Matrix85} of the polynomial $f(x)$ is defined by
$$\mathbf{C}(f)=\begin{bmatrix}
  0 & 0 & \cdots & 0 & -c_0     \\
  1 & 0 & \cdots & 0 & -c_1     \\
  0 & 1 & \cdots & 0 & -c_2     \\
  \vdots& \vdots&\ddots &\vdots &\vdots    \\
  0 & 0 & \cdots & 1 & -c_{d-1}
\end{bmatrix}
.
$$
Since $f(x)$ is irreducible over $\Q[x]$, it has distinct roots $\alpha_1, \alpha_2, \cdots, \alpha_d$. Thus the
companion matrix is diagonalizable as follows:
$$\V\mathbf{C}(f)\V^{-1}=\begin{bmatrix}
\alpha_1 &                      \\
         &\ddots                \\
           &       &\alpha_d
\end{bmatrix}
,
$$
where all the non-diagonal elements are zero and $\V$ represents the
Vandermonde matrix corresponding to the $\alpha_i$:

$$\V=\begin{bmatrix}
 1 & \alpha_1 & \alpha_1^2 & \cdots & \alpha_1^{d-1}  \\
 1 & \alpha_2 & \alpha_2^2 & \cdots & \alpha_2^{d-1}  \\
 \vdots & \vdots  &\vdots &\ddots &\vdots \\
 1 & \alpha_d & \alpha_d^2 & \cdots & \alpha_d^{d-1}
 \end{bmatrix}.
 $$

\subsection{The algorithm and its correctness}
Given a Pisot number $\alpha$ of degree $d$ with conjugates
$\alpha_2,\cdots,\alpha_d$ and its minimal polynomial over $\Q$
$$ f(x)=x^d+c_{d-1}x^{d-1}+\cdots+c_1x+c_0,$$
firstly, we determine $\tau([\alpha^n])$, where $n$ are given positive
numbers.

\vskip 0.2cm

 \noindent \framebox[4.8in]{
\begin{minipage}{4.7in}
{\bf Algorithm 2}

Input: A Pisot number $\alpha$ with conjugates
$\alpha_2,\cdots,\alpha_d$ and its minimal polynomial over $\Q$:
$f(x)=x^d+c_{d-1}x^{d-1}+\cdots+c_1x+c_0$ and a positive integer
$n$.
\begin{enumerate}
\item If $n\leq\log_{|\alpha_2|}{\frac{1}{2(d-1)}},$ compute $[\alpha_1^n]$
directly;
\item If $n>\log_{|\alpha_2|}{\frac{1}{2(d-1)}},$
\begin{enumerate}
\item Construct $\mathbf{C}(f)$;
\item Find a straight-line program for every entry of
       $\mathbf{C}^n(f)$ utilizing the repeated squaring algorithm;
\item Compute the trace of $\mathbf{C}^n(f)$.
\end{enumerate}
\end{enumerate}
Output: A straight-line program of computing $[\alpha^n]$.
\end{minipage}
}

\vskip 0.2cm

Now we proceed to prove Theorem 0.0.5, namely, we need to show that
the proposed algorithm is correct, and the number of basic
operations involved is polynomial in the input size.

\vskip 0.2cm


\begin{proof} ( of Theorem 0.0.5 )
Firstly, we show that the algorithm is correct. When
$n>\log_{|\alpha_2|}{\frac{1}{2(d-1)}}$, by Lemma 3.1, we have
$$[\alpha^n]=\alpha_1^n+\cdots+\alpha_d^n.$$

Since the conjugates of $\alpha$ are distinct, the companion matrix of $f(x)$ can be diagonalized as
$$\V\mathbf{C}(f)\V^{-1}=\begin{bmatrix}
\alpha_1 &                      \\
         &\ddots                \\
           &       &\alpha_d
\end{bmatrix}
,
$$
where all the non-diagonal elements are zero and $\V$ represents the
Vandermonde matrix corresponding to the $\alpha_i$. We have
$$(\V\mathbf{C}(f)\V^{-1})^n=\begin{bmatrix}
\alpha_1^n &                      \\
         &\ddots                \\
           &       &\alpha^n_d
\end{bmatrix}.
$$

Because
$$(\V\mathbf{C}(f)\V^{-1})^n=\V\mathbf{C}^n(f)\V^{-1},$$
we have
$$
\begin{aligned}
\operatorname{tr}(\V\mathbf{C}^n(f)\V^{-1})&=\operatorname{tr}((\V\mathbf{C}(f)\V^{-1})^n)\\
                 &=\alpha_1^n+\cdots+\alpha_d^n\\
                 &=[\alpha^n],
\end{aligned}
$$
where $tr$ is the trace function of the matrix. Furthermore, we have
$$\operatorname{tr}(\mathbf{C}^n(f))=\operatorname{tr}(\V\mathbf{C}^n(f)\V^{-1}),$$ hence
$$\operatorname{tr}(\mathbf{C}^n(f))=[\alpha^n].$$

Next, we analyze the number of basis operations needed. Since the
computation of the the matrix $C^n(f)$ takes $O( \log n )$
matrix multiplications  and other steps take constant number of operations,
we have
$$ \tau ([\alpha^n]) = O( \log n ). $$

\end{proof}

We can modify the last algorithm to compute the modular exponentiation
of a Pisot number as follows:

\vskip 0.2cm

\noindent \framebox[4.8in]{
\begin{minipage}{4.7in}
{\bf Algorithm 3}

Input: A Pisot number $\alpha$ of degree $d$
given by its minimal polynomial over $\Q$:
$f(x)=x^d+c_{d-1}x^{d-1}+\cdots+c_1x+c_0$,
two positive integers
$m,n$.
\begin{enumerate}
\item Construct a straight-line program of length $O(\log n)$
for  $[\alpha^n]$;
\item Evaluate the straight-line program in the ring $\Z/m\Z$.
\item Output the last step of the straight-line program.
\end{enumerate}
Output: $[\alpha^n]\mod{m}$.
\end{minipage}
}

\vskip 0.2cm

\vskip 0.2cm

Sketch of the proof of Corollary 0.0.6: we need to show that the
proposed algorithm is correct and it runs in polynomial time of the
input size. The proof is similar with the proof of Theorem 0.0.5
except that here we compute $\mathbf{C}^n(f)\mod{m}$ instead of $\mathbf{C}^n(f)$
which makes it run in time $(\log (mn))^{O(1)}$.

\section{Concluding remarks}
In this paper, we present two deterministic polynomial time
algorithms about certain computations of Pisot numbers. The first
one is to search a Pisot number $\alpha$ such that $\Q[\alpha]=\F$
given a real Galois extension $\F$ of $\Q$ with integral basis. We
remark that we can find Pisot numbers with high degree utilizing the
algorithm. The second one is to compute the modular exponentiation of a
Pisot number.

\bibliographystyle{plplain}
\bibliography{pisot_comp}
\bibliographystyle{plain}

\end{document}